\documentclass[12pt]{amsart}

\usepackage{amsfonts,amsmath,latexsym,amssymb,verbatim,amsbsy,amsthm}

\usepackage[top=1in, bottom=1in, left=1in, right=1in]{geometry}

\usepackage[dvipsnames]{xcolor}

\usepackage[colorlinks=true, pdfstartview=FitV, linkcolor=RoyalBlue,citecolor=ForestGreen, urlcolor=blue]{hyperref}

\numberwithin{equation}{section}

\renewcommand{\geq}{\geqslant}

\renewcommand{\leq}{\leqslant}

\theoremstyle{plain}
\newtheorem{THEOREM}{Theorem}[section]

\newtheorem{theorem}[THEOREM]{Theorem}

\newtheorem{lemma}[THEOREM]{Lemma}

\newtheorem{assumption}[THEOREM]{Assumption}

\theoremstyle{definition}

\theoremstyle{remark}

\newtheorem{remark}[THEOREM]{Remark}

\newcommand{\myr}[1]{{\color{red}{#1}}} 

\def \a {\alpha} 

\def \d {\delta}

\def \g {\gamma}

\def \e {\varepsilon}
\def \f {\varphi}
\def \F {\Phi}

\def \L {\Lambda}
\def \n {\nabla}

\def \r {\rho}
\def \th {\theta}

\def \CTm {\eta_-}
\def \CTp {\eta_+}
\def\bu{{\mathbf u}}

\def \bk {{\bf k}}
\def \bx {{\mathbf x}}
\def \by {{\mathbf y}}
\def \bu {{\mathbf u}}
\def \bv {{\mathbf v}}
\def \bz {{\mathbf z}}

\def \M {M_0} 

\def \Lal {\L_\alpha}

\def \aL {{\mathcal L}}
\def \aT {{\mathcal T}}

\newcommand{\Z}{\ensuremath{\mathbb{Z}}}   
\newcommand{\R}{\ensuremath{\mathbb{R}}}   
\newcommand{\T}{\ensuremath{\mathbb{T}}}   

\def \p {\partial}

\def \ss {\subset}

\newcommand{\upp}[1]{{#1}_+}
\newcommand{\low}[1]{{#1}_-}

\def \orho {\overline{\rho}}
\def \hf {\frac{1}{2}}

\def \urho  {\rho_-}
\def \orho {\rho_+}
\def \phinf {|\phi|_\infty}

\begin{document}

\title{Eulerian dynamics with a commutator forcing}

\author{Roman Shvydkoy}
\address{Department of Mathematics, Statistics, and Computer Science, M/C 249,\\
University of Illinois, Chicago, IL 60607, USA}
\email{shvydkoy@uic.edu}

\author{Eitan Tadmor}
\address{Department of Mathematics, Center for Scientific Computation and Mathematical Modeling (CSCAMM), and Institute for Physical Sciences \& Technology (IPST), University of Maryland, College Park\newline
Current address: Institute for Theoretical Studies (ITS), 
ETH, Clausiusstrasse 47, CH-8092 Zurich, Switzerland}
\email{tadmor@cscamm.umd.edu}

\date{\today}

\subjclass{92D25, 35Q35, 76N10}

\keywords{flocking, alignment, fractional dissipation, Navier-Stokes, critical thresholds.}

\thanks{\textbf{Acknowledgment.} Research was supported in part by NSF grants DMS16-13911, RNMS11-07444 (KI-Net) and ONR grant N00014-1512094 (ET) and by NSF grant DMS 1515705 (RS). RS thanks CSCAMM for the hospitality for his March 2016 visit which initiated this work. Finally we thank the Institute for Theoretical Studies (ITS) at ETH-Zurich for the hospitality.}

\begin{abstract}
We study a general class of Euler equations driven by a forcing with a \emph{commutator structure} of the form $[\aL,\bu](\rho)=\aL(\rho \bu)- \aL(\rho)\bu$, where $\bu$ is the velocity field and $\aL$ is the   ``action'' which belongs to a rather general class of translation invariant operators. Such systems arise, for example,   as the hydrodynamic description of velocity alignment, where  action involves convolutions with bounded, positive influence kernels, $\aL_\phi(f)=\phi*f$. 
Our interest lies with a much larger class of $\aL$'s which are neither bounded nor positive.

In this paper we develop a global regularity theory in the one-dimensional setting, considering three prototypical sub-classes of actions. We  prove  global regularity for \emph{bounded} $\phi$'s which otherwise are allowed to change sign. Here we derive sharp critical thresholds such that sub-critical initial data $(\rho_0,u_0)$ give rise to global smooth solutions. Next, we study \emph{singular} actions associated with  $\aL=-(-\partial_{xx})^{\alpha/2}$, which embed the fractional Burgers' equation of order $\alpha$. 
We prove global regularity for $\alpha\in [1,2)$. Interestingly, the singularity of the fractional kernel $|x|^{-(n+\alpha)}$, avoids an initial threshold restriction. Global regularity of the critical endpoint $\alpha=1$ follows with double-exponential $W^{1,\infty}$-bounds. Finally, for the other endpoint $\alpha=2$, we prove the global regularity of the Navier-Stokes equations with density-dependent viscosity associated with the \emph{local} $\aL=\Delta$.
\end{abstract}

\maketitle
\setcounter{tocdepth}{1}
\tableofcontents

\section{Fundamentals. Euler equations with a commutator structure}
We are concerned with  a new class of Eulerian dynamics where a  velocity field, $\bu : \Omega\times \R_+ \mapsto  \R^n$, is driven  by the system
\begin{equation}\label{e:main}
\left\{
\begin{split}
\rho_t + \n \cdot (\rho \bu) & = 0, \\
\bu_t + \bu \cdot \n \bu &= \aT(\rho, \bu),
\end{split} \right. \qquad (\bx,t)\in \Omega\times \R_+.
\end{equation}
The main feature here is the commutator structure of  the forcing
\begin{equation}\label{eq:L}
\aT(\rho \bu) = [\aL, \bu](\rho):=\aL(\rho\bu)-\aL(\rho)\bu,
\end{equation}
expressed in terms of a self-adjoint operator  $\aL: \R \mapsto \R$ (the action on $\rho\bu$ is interpreted component-wise). We focus on the Cauchy problem over the whole space $\Omega=\R^n$ or over the torus $\Omega={\mathbb T}^n$.

A typical example  is provided by radial mollifiers, $\aL(f)=\phi*f$, associated with  integrable $\phi \in L^1$, which yields the commutator forcing
\begin{equation}\label{eq:phi}
\aT(\rho, \bu)(\bx) = \phi*(\rho\bu)- (\phi*\rho)\bu = \int_{\R^n}\phi(|\bx-\by|)  (\bu(\by) - \bu(\bx) )\rho(\by)d\by.
\end{equation}
The corresponding system \eqref{e:main},\eqref{eq:phi} arises as 
macroscopic realization of the Cucker-Smale agent-based dynamics \cite{CS2007a,CS2007b}, which describes the collective motion of $N$ agents, each of which adjusts its velocity to a weighted average of  velocities of its neighbors dictated by an \emph{influence function} $\phi$, 
\begin{align*}
\left\{\begin{array}{ll}\dot{\bx}_i&=\bv_i, \\
\dot{\bv}_i&=\displaystyle \frac{1}{N}\sum_{j=1}^N \phi(|\bx_i-\bx_j|)(\bv_j-\bv_i),
\end{array}\right.\qquad (\bx_i,\bv_i)\in \R^n\times \R^n.
\end{align*}
For large crowds, $N\gg 1$, one is led to the hydrodynamic description \eqref{e:main},\eqref{eq:phi}, \cite {HT2008,CCP2017}. For recent results which justify the passage  to Cucker-Smale kinetic and hydrodynamic descriptions with \emph{weakly singular} kernels $\phi$ (of order $< \hf$) we refer to \cite{Pes2015,PS2016}.
The global regularity of  such one- and two-dimensional systems  \eqref{e:main},\eqref{eq:phi} was studied in \cite{TT2014, CCTT2016,HT2016}.
For  \emph{bounded, positive} mollifiers it was shown that there exist  certain \emph{critical thresholds} in the phase space of initial configurations, $(\rho_0 >0,\bu_0)$, such that sub-critical initial data  propagate the initial smoothness of $(\rho(\cdot,0), \bu(\cdot,0) )=(\rho_0,\bu_0)$ globally in time.

Our interest lies in the global regularity of \eqref{e:main},\eqref{eq:L} for a much larger class of $\aL$'s which are neither positive nor bounded. We have three typical examples in mind.

\subsection{Examples}
Consider $\aL=\aL_\phi$ of the form 
\begin{equation}\label{eq:Lphi}
\aL_\phi(f)(\bx):= \int_{\R^n} \phi(|\bx-\by|) \big(f(\by)-f(\bx)\big)d\by.
\end{equation}
Our first example involves \emph{bounded} kernels with a finite positive mass,  denoted $\phi\in  L^\infty_\#:=\{ \phi \in L^\infty  \ | \ 0 <\int \phi(r)dr <\infty \}$, but otherwise are allowed to \emph{change sign}.
The resulting commutator $\aT_\phi=[\aL_\phi,\bu](\rho)$ coincides with the usual convolution action  in \eqref{eq:phi},
\begin{equation}\label{eq:Tphi}
\aT(\rho,\bu)(\bx)=[\aL_\phi,\bu](\rho)(\bx) = \int_{\R^n} \phi(|\bx-\by|)(\bu(\by)-\bu(\bx))\rho(\by)d\by, \qquad \phi\in  L^\infty_\#.
\end{equation}

The action $\aL_\phi$ in \eqref{eq:Lphi} and its commutator forcing  \eqref{eq:Tphi} are well defined for \emph{non-integrable} $\phi$'s as well. As a second example we consider, $\phi_\alpha(\bx):=|\bx|^{-(n+\alpha)}$, associated with the action of the fractional Laplacian\footnote{We shall abuse  notations by abbreviating $\aL_{\phi_\alpha}:=\aL_\alpha$ since the distinction is clear from the context of the sub-index involved.}  $\aL_{\alpha}(f)=-\L_{\alpha}(f), \ \alpha<2$,
\[
\L_{\alpha}(f)(\bx)= p.v. \int_{\R^n} \frac{f(\bx)-f(\by)}{|\bx-\by|^{n+\alpha}}d\by, \qquad \L_{\alpha}=(-\Delta)^{\alpha/2}, \quad 0<\alpha <2.
\]
The corresponding forcing is then given by the singular integral
\begin{equation}\label{e:T}
\aT(\rho, \bu)(x) = -\L_{\alpha} (\rho \bu)+ \L_{\alpha}(\rho) \bu= p.v. \int_{\R^n} \frac{  \bu(\by) - \bu(\bx) }{|\bx-\by|^{n+\alpha}}\rho(\by) d\by.
\end{equation}
The operator $\aT$ in \eqref{e:T} is well-defined as a distribution over the whole space $\Omega=\R^n$. When dealing with the torus $\Omega=\T^n$, the forcing $\aT$  can be expressed in terms of the periodized kernel $\phi_\alpha(\bz) = \sum_{\bk\in \Z^n} \frac{1}{|\bz+2\pi \bk|^{n+\alpha}}$. 
 
\ifx
\myr{Should  add here a note(?) observing that the  formal limit $\rho \rightarrow 1$ yields}\footnote{... the incompressible Navier-Stokes equations ($n\geq2$) and the corresponding Burgers equation ($n=1$) with fractional-order diffusion, e.g., \myr{Need references on fractional NS} \cite{BT2014,IJS2016}.}.
\fi

Finally, as a third example we consider  the full Laplacian $\aL=\Delta$ corresponding to the limiting case $\alpha=2$ with forcing $\aT (\rho, \bu)=\rho\Delta \bu + 2(\nabla \rho\cdot \nabla) \bu$. This leads to the density-dependent system of pressureless Navier-Stokes equations
\begin{equation}\label{e:NSs}
\left\{
\begin{split}
\rho_t + \n \cdot (\rho \bu) & = 0,\\
(\rho \bu)_t + \n (\rho \bu \otimes \bu)  &= \n (\rho^2 D\bu), \qquad D\bu=\{\partial_i u_j\}.
\end{split}\right.
\end{equation}

\smallskip\noindent
We close by noting that these equations are typically  come ``equipped'' with certain standard global bounds.  Thus, in addition to the obvious conservation of mass,
\[
M_0:=\int \rho_0(\bx)d\bx\equiv \int \rho(\bx,t)d\bx,
\]
we have, since $\aL$ is assumed self-adjoint, $\int\big(\rho \aL(\rho \bu) - \aL(\rho)\rho \bu\big) d\bx=0$, conservation of momentum,
$\int_{\R^n } \rho \bu(\cdot,t) \,d\bx = \int_{\R^n } \rho_0 \bu_0 \, d\bx$.
Also for $\aL_\phi$ we have the  $\rho$-weighted energy-enstrophy bound
\begin{equation}\label{e:en}
\int_{\R^n\times \{T\} } \rho |\bu|^2 d\bx + \int_0^T \int_{\R^n\times\R^n} \rho(\bx) \rho(\by) \phi(|\bx-\by|)|\bu(\bx) - \bu(\by)|^2  d\bx d\by dt =  \int_{\R^n} \rho_0 |\bu_0|^2 d\bx.
\end{equation}

\subsection{The one dimensional case. Statement of main results.}\label{sec:1d}
The main focus of this paper  is one-dimensional case where \eqref{e:main},\eqref{eq:L} reads,
\begin{equation}\label{eq:1D}
\left\{
\begin{split}
\rho_t + (\rho u)_x & = 0,\\
(\rho u)_t + (\rho u^2)_x &= \rho \aL(\rho u)- \rho \aL(\rho)u,
\end{split}\right. \qquad (x,t)\in \Omega\times \R_+.
\end{equation}
We shall make a detailed study on the propagation of  regularity of \eqref{eq:1D} for sub-critical initial data, dictated by the properties of $\aL$.

We begin by recalling that \eqref{eq:1D} with $\aL=\aL_\phi$ amounts to the  one-dimensional  Cucker-Smale ``flocking hydrodynamics'' \cite{CS2007a,CS2007b,MT2014,CCP2017}
\begin{equation}\label{eq:1DCS}
\left\{
\begin{split}
\rho_t + (\rho u)_x & = 0,\\
 (\rho u)_t + (\rho u^2)_x &= \int_{\R} \phi(|x-y|) \big(u(y) - u(x) \big)\rho(x)\rho(y)dy,
\end{split}\right.\qquad (x,t)\in \Omega\times \R_+.
\end{equation}
  Global regularity  for bounded \emph{positive} $\phi$'s  persists if and only if the initial data are \emph{sub-critical} in the sense that  \cite{CCTT2016}
\begin{equation}\label{eq:1DCSCT}
u'_0(x)+\phi*\rho_0(x) \geq 0 \quad \text{for all} \ \ x\in \R.
\end{equation}
 
In section \ref{sec:bounded} we extend this regularity result for general bounded $\phi$'s whether positive or not. The results below are stated over the torus, $\Omega=\T^1$,  for the purely technical reason of  securing a \emph{uniform} lower  bound of the density away from vacuum, which in turn provides uniform parabolicity of the $u$-equation.  However, the local well-posedness follows from our analysis over $\Omega=\R$ line as well.

\begin{theorem}\label{thm:bdd}
Consider the hydrodynamics flocking model \eqref{eq:1DCS}  with a bounded mollifier, $\phi\in  L^\infty_\#$  having a positive total mass $I(\phi)=\int \phi(r)dr >0$, and subject to sub-critical initial data $(\rho_0,u_0)\in (L^1_+(\T^1), W^{1,\infty}(\T^1))$, such that 
\[
u'_0(x) +\phi*\rho_0(x) >0, \qquad  x\in \T^1.
\]
Then \eqref{eq:1DCS} admits global smooth solution. 
\end{theorem}

Next, we extend this result to the case of the positive \emph{singular} mollifiers $\phi_\alpha(r)=|r|^{-(1+\alpha)}$ 
\begin{equation}\label{e:FNS}
\left\{
\begin{split}
\rho_t + (\rho u)_x & = 0,\\
(\rho u)_t + (\rho u^2)_x &= p.v. \int_{\R} \frac{ u(y) - u(x) }{|x-y|^{1+\alpha}}\rho(x)\rho(y) dy, \quad \alpha < 2,
\end{split}\right.\qquad (x,t)\in \Omega\times \R_+.
\end{equation}
Here we follow a general iteration scheme for proving (higher) regularity outlined in section \ref{sec:scheme}, in which one seeks bounds on the density, $\rho$, and then bounds the ``action'' $\aL(\rho)$. The uniform bounds on the density for all three cases are worked out in section \ref{sec:bddensity}.
 We then turn to secure bounds on the action or --- what amounts to the same thing, uniform bound on $u_x$, which in turn yields global well-posedness.
  In section \ref{sec:boundedL} we discuss the global regularity for bounded mollifiers $\aL=\aL_\phi$ in \eqref{eq:1DCS}, and in section \ref{sec:NSE} for the the Navier-Stokes equations,
 $\aL=\Delta$, 

 \begin{equation}\label{e:NS}
\left\{
\begin{split}
\rho_t +  (\rho u)_x & = 0,\\
(\rho u)_t + (\rho u^2)_x &= (\rho^2 u_x)_x,
\end{split}
\right.\qquad (x,t)\in \Omega\times \R_+.
\end{equation}
This is the one-dimensional special case of the general class of Navier-Stokes equations studied in \cite{BDGV2007}.
\begin{theorem}\label{thm:ns}
Consider the Navier-Stokes equations \eqref{e:NS} subject to
 initial data  $(u_0,\rho_0)\in H^2(\T^1) \times H^3(\T^1)$.
Then \eqref{e:NS} admits a global solution in the same class. 
\end{theorem}
Finally, in section \ref{sec:singularaction} we prove the global smooth solutions for the commutator forcing associated with the \emph{singular} action $\aL_\alpha(\rho)=-\Lal(\rho)$  corresponding to singular kernel $\phi_\alpha(r)=|r|^{-(1+\alpha)}, \ 1\leq \alpha<2 $.
\begin{theorem}\label{thm:singular}
Consider the system of equations \eqref{e:FNS} with $1 \leq \alpha <2$ subject to
 initial data $(u_0,\rho_0)\in H^3(\T^1) \times H^{2+\a}(\T^1)$.
Then \eqref{e:FNS} admits a global solution in the same class. 
\end{theorem}

It is remarkable that the singularity of $\phi_\alpha=|x|^{-(1+\alpha)}$ removes the requirement for a finite critical threshold which is otherwise called for integrable $\phi\in L^\infty_\#$. Specifically,  in section \ref{sec:singulardensity} we prove  that for any singular kernel such that $\lim_{r\downarrow 0[\text{mod}\, 2\pi]} r\cdot \min_{|z|\leq r} \phi(|z|) \uparrow \infty$, the density of the corresponding system \eqref{eq:1DCS} remains \emph{uniformly} bounded which in turn drives the global regularity. The analysis of equation with the singular action $\aL_\alpha$ becomes critical when $\a$ reaches value $1$. The necessary $W^{1,\infty}$-bounds on the solution pair $(u,\rho)$ in this case admit double-exponential growth in time, consult \eqref{eq:double}.

\section{Propagation of global regularity. A general iteration scheme.}\label{sec:scheme} 

\subsection{$L^\infty$-bound of the velocity} 
 We assume that $\aL$ satisfies the following \emph{monotonicity condition}.
Let $\displaystyle \upp{x}=\mathop{\text{arg\,max}}_x g(x)$ and $\displaystyle \low{x}=\mathop{\text{arg\,min}}_xg (x)$. Then for $f\geq 0$
\begin{equation}\label{eq:mon}
\left\{
\begin{split}
 \aL(f g)(\upp{x}) & \leq \aL(f)(\upp{x})g(\upp{x}), \qquad  g(\upp{x})=\max_x g(x)\\
  \aL(f g)(\low{x}) & \geq \aL(f)(\low{x})g(\low{x}), \qquad g(\low{x})=\min_x g(x)
  \end{split}
  \right.
\end{equation}
which holds for $\aL=\aL_\phi$ with positive $\phi$'s. Application of \eqref{eq:mon} with $(f,g)=(\rho,u)$ implies that 
\[
\aT(\rho,u)(\upp{x}) \leq 0 \leq \aT(\rho, u)(\low{x}), \qquad x_\pm= \left\{\begin{split}
&\mathop{\text{arg\,max}} \, u(\cdot,t)\\
&\mathop{\text{arg\,min}} \, u(\cdot,t)\end{split}\right.
\]
 and  yields that $u$ in \eqref{eq:1D} (and likewise --- the velocity components  $u_i$ in \eqref{e:main}) satisfy maximum/minimum principle 
\begin{equation}\label{max}
\min_\bx u_i(\bx,0) \leq u_i(\bx,t)  \leq \max_\bx u_i(\bx,0) 
\end{equation}
Likewise, $\|\bu(\cdot,t)\|_{L^\infty}$ remains finite for $\phi\in  L^\infty_\#$ with arbitrary sign.  

\subsection{Critical threshold and a first order conservation law}\label{sec:CT} 
We outline our general strategy for  tracing the  global regularity of \eqref{eq:1D}.
The key observation is that the commutator form of \eqref{eq:1D} entails the transport of $u_x+\aL(\rho)$  away from vacuum. To this end, differentiate  \eqref{e:main} to find that  $u':= u_x$ satisfies
\begin{equation}\label{}
u'_t + u u'_x + (u')^2 = \aL(\rho u)_x -u\aL(\rho)_x- u' \aL(\rho).
\end{equation}
For the latter we use the density equation, $\aL(\rho u)_x= \aL\big((\rho u)_x\big)=- \aL(\rho)_t$  to conclude
\[
(u' + \aL(\rho))_t + u(u' + \aL(\rho))_x + u' (u' + \aL(\rho)) = 0.
\]
This calls for introduction of the new variable, $e := u' + \aL(\rho)$, which is found to satisfy 
\begin{equation}\label{e:e}
e_t + (u e)_x = 0, \qquad e=u'+\aL(\rho).
\end{equation}
Together with the density equation, this yields that $e/\rho$ is governed by the transport equation
\begin{equation}\label{eq:trans}
\left(\frac{e}{\rho}\right)_t +u\left(\frac{e}{\rho}\right)_x=0.
\end{equation}
Hence $e/\rho$ remains constant along the
 characteristics $\dot{x}(t) = u(x(t),t)$,
\begin{equation}\label{char}
\frac{e(x(t),t)}{\rho(x(t),t)} = \frac{e_0(x)}{\rho_0(x)}. 
\end{equation}
It follows that if $e_0/\rho_0$ is allowed to have singularities, then these initial singularities will propagate along characteristics and a solution of \eqref{eq:1D} will consist of strips of regularity trapped between the curves carrying these singularities.   
To avoid this scenario,  calls for the following bound to hold.
\begin{assumption}{[Critical threshold]}\label{ass:CT} There exist finite constants $\CTm \leq 0 < \CTp$ such that
\begin{equation}\label{eq:bound}
\CTm \leq \frac{e_0(x)}{\rho_0(x)} \leq \CTp \ \ \text{for all} \ \  x\in \Omega.
\end{equation}
\end{assumption}
\begin{remark}
We note in passing that integration of \eqref{eq:bound} yields $\low{\eta}\M \leq \int \big(u'_0 + \aL(\rho_0)\big) dx$.  Hence, since $\aL_\phi(\rho_0)$ has zero mean and  $u(\cdot,t)$ is either periodic or assumed to have  vanishing  far-field boundary values,  it follows that \eqref{eq:bound} requires $\low{\eta}\leq 0$.
\end{remark}
 
 \medskip
We will investigate the propagation of regularity of solutions subject to sub-critical initial data \eqref{eq:bound}.
\subsection{The iteration scheme --- a priori control estimates via $e$} 
The study of global well-posedness for all three cases of commutator forcing we have in mind --- bounded, singular and local (NS) mollifiers, share a common scheme of establishing control over the key quantities, even though the handling of the three cases is quite different when it comes to analytic details. In this section we highlight those main common features in three steps.

\medskip\noindent
$\bullet$ Step \#1 (\emph{Pointwise bounds on the density}). Our aim is to show that for a certain range of threshold bounds $\CTm \leq 0 <\CTp$, the density remains bounded from above and away from the vacuum
\begin{equation}\label{eq:rhob}
0< \rho_- \leq \rho(\cdot,t) \leq \upp{\rho} < \infty.
\end{equation}
In view of transportation of the ratio $e/\rho$, \eqref{char}, we also have
\begin{equation}\label{eq:pointwise}
\CTm \leq  \frac{e(\cdot,t)}{\rho(\cdot,t)} \leq \CTp, \qquad \CTm\leq0.
\end{equation}
We conclude that the quantity of interest, $e=u_x+\aL(\rho)$, will remain uniformly bounded, 
\[
\low{e}:=\CTm\upp{\rho} \leq e(\cdot,t) \leq \upp{e}:=\CTp\upp{\rho}.
\]

\noindent
$\bullet$ Step \#2 (\emph{Pointwise bound on the action $\aL(\rho)$ and slope $u_x$}). 
Equipped with the uniform bound on $e$ we turn to establish a bound on the action $\aL(\rho)$, which is equivalent to controlling the slope $u_x$.
In the case of bounded mollifiers we seek a pointwise bound on the action $\aL(\rho)$
\begin{equation}\label{eq:Lb}
\low{\aL}\leq \aL(\rho) \leq \upp{\aL}, \qquad \rho\in L^1_+\cap L^\infty.
\end{equation}
This will imply the desired $C^1$-bound of the velocity 
\[
\CTm\upp{\rho}-\low{\aL} \leq u_x(\cdot, t) \leq \CTp\upp{\rho}+\upp{\aL}.
\]
For singular fractional mollifiers $\aL_\alpha$, we focus on the critical case $\a=1$, where we  use a nonlocal maximum principle to establish control over $\rho'$ which in turn enables us to control $u_x$ indirectly, thus avoiding an additional  obstacle coming from the Hilbert transform. For the NS case, we first control the slope $u_x$ via energy bounds, then conclude with control of $\aL(\r)= \r_{xx}$.

It is clear from the fact that the higher-order quantity $e$ satisfies lower-order estimates that a proper statement of well-posendess result for singular mollifiers requires $\rho$ to be in a regularity class $X^{s+\a}$ provided $u$ is in the class $X^{s+1}$, while $e$ is in the class $X^s$.

\medskip \noindent
$\bullet$ Step \#3 (\emph{Higher regularity control}).  The necessary bounds sought in \eqref{eq:rhob},\eqref{eq:Lb} may require a restricted set of initial configurations depending on finite critical threshold assumed in \eqref{eq:bound}. Whether these thresholds $\eta_\pm$ are restricted or not, the corresponding bounds will be derived solely on the basis of  the mass equation for $\rho$, and the fact that $e=u_x+\aL(\rho)$ satisfies the transport equation \eqref{eq:trans}. This argument can be iterated to higher derivatives as follows.
 Note that if  a quantity $Q$ is transported, $Q_t + u Q_x = 0$, then the \emph{same} transport equation governs $Q_x/\rho$ 
\begin{equation}\label{e:tranQ}
\left(\frac{Q_x}{\rho}\right)_t+u\left(\frac{Q_x}{\rho}\right)_x=0.
\end{equation}
 Let us apply this argument to $Q = e/\rho$: then if  $|(e/\rho)_x|/\rho$ is bounded at $t=0$ it will remain bounded at later time. Unraveling  the formulas, we obtain the pointwise bound
\begin{equation}\label{e:control1}
|e'(x,t)| \leq C(e_\pm,\rho_\pm) |\rho'(x,t)|.
\end{equation}
This control bound will become a key tool in proving Theorem~\ref{thm:singular}. 

Now that $(e/\rho)_x /\rho$ is transported, we can apply the argument above repeatedly to obtain a hierarchy of pointwise bounds 
\begin{equation}\label{e:high}
|e^{(k)}(x,t)| \leq C |\rho^{(k)}(x,t)|, \qquad k=0,1,\ldots
\end{equation}
It is therefore clear that such bounds would allow to apply the same control principle as stated above in extending our results  into higher order Sobolev spaces. However, we will leave to pursue this direction to a future work.

\section{Bounded density in one-dimensional equations in commutator form}\label{sec:bddensity}
In this section we implement the above strategy for global regularity in the presence of commutator forcing, $\aT_\phi$, depending on the properties of the mollifier  $\phi$.
We begin with a general discussion on the boundedness of the density sought in step \#1.
Here, the bound \eqref{eq:rhob} is driven by the diffusive character of the mass equation, which is revealed once we rewrite the mass equation of \eqref{eq:1D} in the form
\begin{equation}\label{eq:erho}
\rho_t + u\rho_x = -e\rho+\rho \aL(\rho).
\end{equation}
In view of the uniform bound \eqref{eq:pointwise}, we see that $e\rho \sim \rho^2$ behaves as a quadratic term. This  implies
\begin{equation}\label{eq:bdrho}
-\upp{\eta}\rho^2 +\rho \aL(\rho) \leq \rho_t + u\rho_x \leq  -\low{\eta}\rho^2+\rho \aL(\rho).
\end{equation}

We turn to check step \#1 in the three cases of interest.
 
\subsection{Bounded density with bounded mollifiers $\aL_\phi, \, \phi\in L^\infty_\#$}\label{sec:bounded}
Consider the case of  $\aL=\aL_\phi= \int \phi(|x-y|)(\rho(y)-\rho(x))dy$ with $\phi \in L^\infty_\#$ which is assumed to have a positive mass $\int \phi(r)dr>0$. We emphasize that $\phi$ need not be positive. We verify  the boundedness of $\rho$ using the straightforward bound 
\begin{equation}\label{eq:bdL}
-I(\phi)\rho-\phinf\M \leq \aL_\phi(\rho)  \leq -I(\phi)\rho+\phinf\M, \qquad I(\phi):=\int \phi(r)dr>0.
\end{equation}
Inserted into \eqref{eq:bdrho} we find
\[
-\big(\upp{\eta}+I(\phi)\big)\rho^2 -\phinf\M\rho  \leq \rho_t + u\rho_x \leq -\big(\low{\eta}+I(\phi)\big)\rho^2+\phinf\M\rho.
\]
The inequality on the left shows that  along characteristics, the density is bounded away from vacuum by a lower-bound $\rho(t)\gtrsim e^{-\phinf\M t}\rho_0$. The inequality on the right shows that the density remains bounded from above for \emph{any} $\low{\eta} > -I(\phi)$, that is, provided \eqref{eq:bound} holds for such $\low{\eta}$'s,
\begin{equation}\label{eq:CTphib}
u'_0(x) +\phi*\rho_0(x)-I(\phi)\rho_0(x) \geq \low{\eta}\rho_0(x), \qquad \low{\eta} > -I(\phi).
\end{equation}

\subsection{Bounded density with singular mollifiers $\aL_\alpha, \, \alpha<2$}\label{sec:singulardensity}
To bound the density from above, we consider the case of positive  mollifiers  which are singular in the sense that  
\begin{equation}\label{eq:Lsing}
\lim_{r\downarrow 0} rm_\phi(r) \uparrow \infty, \qquad m_\phi(r):=\min_{|z|\leq r} \phi(|z|).
\end{equation}
In this case we use the bound
\[
\begin{split}
\aL(\rho)(\upp{x}) & \leq \int_{|x-y|\leq r} \phi(|x-y|)(\rho(y)-\upp{\rho})dy\\
&  \leq m_\phi(r)\int_{|x-y|\leq r} (\rho(y)-\upp{\rho})dy \leq m_\phi(r)\M- 2rm_\phi(r)\upp{\rho}.
\end{split}
\]
By assumption, for \emph{any} $\low{\eta}\leq 0$ we can choose a small enough $r=\upp{r}$ such that $2\upp{r} m_\phi(\upp{r}) =1-\low{\eta}$ and the bound on the right of \eqref{eq:bdrho} then implies that the maximal value of the density $\upp{\rho}(t)=\rho(\upp{x}(t),t)$ satisfies
\[
\upp{\dot{\rho}} \leq -\low{\eta}\upp{\rho}^2 - 2\upp{r}m_\phi(\upp{r})\upp{\rho}^2 + c_0\upp{\rho} \leq -\upp{\rho}^2 +c_0\upp{\rho}, \qquad c_0=  m_\phi(\upp{r})\M
\]
Thus, $\rho(\cdot,t)$ remains bounded from above. We conclude that for  singular kernels satisfying \eqref{eq:Lsing} , the density remains upper-bounded \emph{independent} of the lower threshold $\CTm$. In particular, this  applies to $\phi_\alpha(r) = r^{-(1+\alpha)}, \alpha<2$.

We turn to the lower bound on the density away from vacuum. For positive $\phi$'s, whether singular or not, we have\footnote{This is a special case of the  monotonicity condition \eqref{eq:mon} with $(f,g)=({\boldsymbol 1}, \rho)$ implies 
$\aL(\rho)(\low{x}) \geq  \aL({\boldsymbol 1}(\low{x}))\low{\rho}=0$.}
\[
\aL_\phi(\rho)(\low{x})=\int_y \phi(|x-y|)(\rho(y)-\rho(\low{x}))\rho(y)dy \geq 0. 
\]
Therefore, the inequality on the left of \eqref{eq:bdrho} implies that  minima values of the density,
$\low{\rho}(t)=\rho(\low{x}(t),t)$ at  \emph{any interior} point $\low{x}(t)=\mathop{\text{arg\,min}}_{|y|\leq R}\{\rho(y,t)\}$ with $|\low{x}| <R$, 
satisfy   $\low{\dot{\rho}} \geq -\upp{\eta}\low{\rho}^2$   and hence $\rho(\cdot,t) >0$.
In the particular case of  the torus $\Omega={\mathbb T}^1$, we conclude with a 
\emph{uniform} lower bound away from vacuum
\begin{equation}\label{eq:away-from-vacuum}
\rho(\cdot,t) \geq \low{\rho}(t)=\frac{\low{(\rho_0)}}{t\upp{\eta}\low{(\rho_0)}+1}, \qquad 
\low{(\rho_0)}=\min_{x\in {\mathbb T}^1}\rho_0(x) >0.
\end{equation}

\subsection{Bounded density with NS equations $\aL_2=\partial_{xx}$}\label{sec:nsdensity}
We use the  regularization coming from te parabolic part of the mass equation 
which  becomes evident when \eqref{e:NS} is written in the form
\begin{equation}\label{eq:nsagain}
\rho_t+u\rho_x +e\rho= \rho \rho_{xx}.
\end{equation}
It implies the lower-bound which could be read from the LHS of \eqref{eq:bdrho},
$\low{\dot{\rho}} \geq -\upp{\eta}\low{\rho}^2$ which recovers the same lower-bound \eqref{eq:away-from-vacuum}. Trying to pursue  the same argument  for an upper-bound of the density fails when using the RHS of \eqref{eq:bdrho}. Instead, we note that the quantity $f:=u+\rho_x$  is the primitive of $e$ and hence satisfies the transport equation
$f_t+uf_x=0$. This follows by direct computation of \eqref{eq:1D} 
\[
u_t+uu_x  =-(\rho'_t + u\rho'_x) .
\]
 It follows that 
\begin{equation}\label{e:rhox}
|\rho'(\cdot,t)| \leq 2|u_0|_\infty + |\rho'_0|_\infty.
\end{equation}
Here and throughout $|\cdot|_p$, $1\leq p \leq \infty$, denotes the $L^p$-norm. Given the uniform bound on $\rho'$ and since we already proved that $\rho >0$, \eqref{e:rhox} ties the upper bound for $\rho$ as well.  We note in passing that even though we can now express the density equation as a \emph{pure} diffusion
\begin{equation}\label{eq:pure}
\rho_t =\rho\rho_{xx} + F,
\end{equation}
with bounded forcing $ F=-u\rho_x -\rho e \in L^\infty$, we can only reach the end-point Schauder estimate $\rho_{xx} \in BMO$ (see \cite{Schl1996}), which is not enough to secure a uniform bound  of $\aL(\r)=\r_{xx}$ necessary to get control over the slope $u_x$. We will provide additional details how to reach that bound, which is needed for the global existence of NS equations in section \ref{sec:NSE} below.

\section{Global existence: bounded mollifiers, $\aL=\aL_\phi$}\label{sec:boundedL}
With regard to Theorem~\ref{thm:bdd},  it is straightforward to verify step \#2 in the case of bounded mollifiers --- in view of \eqref{eq:bdL}, the upper-bound of $\rho$ implies that $\aL(\rho)$ is uniformly bounded, $|\aL_\phi(\rho)+I(\phi)\upp{\rho}| \leq \phinf\M$, and hence $u_x=e-\aL(\rho)$ is uniformly bounded. We conclude the global regularity for sub-critical initial data satisfying \eqref{eq:CTphib}, namely, for a fixed $\epsilon>0$ there holds 
\[
u'_0(x) +\phi*\rho_0(x) \geq \epsilon\rho_0(x), \qquad \epsilon >0.
\]
In the particular case of $\T^1$, this requires the positivity of $u'_0+\phi*\rho_0$ stated in theorem \ref{thm:bdd}. This recovers the same critical threshold of positive mollifiers \eqref{eq:1DCSCT}.

\section{Global existence: Navier-Stokes equations, $\aL=\Delta$}\label{sec:NSE}
In this section we will prove Theorem~\ref{thm:ns}. Recall that the boundedness of $\rho$ conveys to boundedness of $e=u_x+\rho_{xx}$, via
\begin{equation}\label{eq:uxbd}
|e(x,t)| \leq \eta |\rho(x,t)|,
\end{equation}
and that $\rho$ satisfies further a priori $C^1$-regularity \eqref{e:rhox}. Note that these low-regularity a priori bounds  hold classically under the assumptions  $u\in H^2$, $\rho \in H^3$, which are the spaces for which Theorem~\ref{thm:ns} is stated, and these are the lowest integer regularity $H^n$-spaces that justify the above computations. We now proceed by establishing a priori estimates in these spaces.

First, let us quantify control over the high-order regularity of $e$.

\begin{lemma}\label{l:e} For each $n =0,1,...$ we have the following a priori estimate
\begin{equation}\label{eq:leb} 
\p_t |e|_{H^n}^2 \leq C (|e|_{H^n}^2 + |u|_{H^{n+1}}^2)( |u_x|_\infty + |e|_\infty).
\end{equation}
\end{lemma}
\begin{proof}
For $n=0$ the Lemma follows easily by testing the $e$ equation \eqref{e:e}.  For $n=1,...$, let us differentiate \eqref{e:e} $n$ times and test with $e^{(n)}$. 
We obtain (dropping the integral signs)
\[
\p_t |e^{(n)}|_{2}^2 \lesssim u e^{(n+1)} e^{(n)} + \sum_{k=1}^{n+1} u^{(k)} e^{(n+1 - k)} e^{(n)}.
\]
For the first term we integrate by parts to obtain trivially $ | u e^{(n+1)} e^{(n)} | \lesssim |u_x|_\infty |e|_{H^n}^2 $.  
For each of the remaining terms on the right, $k=1,\ldots, n+1$, we apply Gagliardo-Nirenberg inequalities,
$| \p^i f |_{\frac{2n}{i}} \leq |f|_\infty^{1- \frac{i}{n}} |f|_{H^n}^{\frac{i}{n}}, \ 1\leq i \leq n$, obtaining
\[
\begin{split}
| u^{(k)} e^{(n+1 - k)} e^{(n)} | & \leq |e^{(n)} |_2 |u_x^{(k-1)}|_{\frac{2n}{k-1}} |e^{(n+1-k)}|_{\frac{2n}{n+1-k}} \\
& \leq |e^{(n)} |_2 | u_x |_\infty^{1- \frac{k-1}{n}} | u_x |_{H^n}^{\frac{k-1}{n} } | e |_\infty^{1- \frac{n+1- k}{n}} |e|_{H^n}^{\frac{n+1- k}{n} } \\
&\leq |e^{(n)} |_2^{\frac{2n + 1 - k}{n}} | u_x |_{H^n}^{\frac{k-1}{n} } | u_x |_\infty^{1- \frac{k-1}{n}} | e |_\infty^{1- \frac{n+1- k}{n}}
\end{split}
\]
and by Young's inequality 
$| u^{(k)} e^{(n+1 - k)} e^{(n)} | \leq (|e|_{H^n}^2 + |u|_{H^{n+1}}^2)( |u_x|_\infty + |e|_\infty)$ which completed the proof.
\end{proof}

We now proceed establishing bounds on $u_x$ and $u_{xx}$ in a sequence of increasing norms, which eventually will close the estimates together with Lemma~\ref{l:e}. Recall
\begin{equation}\label{e:uagain} 
u_t + u u' = \rho u'' + 2 \rho' u'.
\end{equation}
Testing with $u$ and using \eqref{e:rhox}  we obtain 
\[
\p_t |u|_2^2 = - \int \rho |u'|^2 + \int \rho' u u' \leq - \frac12 \int \rho |u'|^2 + 2\rho_-^{-1} |u|_2^2.
\]
This proves the natural energy  bound $u\in L^\infty_t L^2_x \cap L^2_t H^1_x$.  Next, we test with $ - u''$ to obtain (dropping the integrals)
\[
\p_t |u'|_2^2 = |u u' u''| - \rho |u''|^2 + 2 |\rho' u' u''| \lesssim - \frac12  \rho |u''|^2 + |u|^2|u'|^2 + |\rho'|^2|u'|^2.
\]
Using uniform bound on $u$ and \eqref{e:rhox},
\[
\p_t |u'|_2^2 \leq - \frac12  \rho_- |u''|_2^2 + C |u'|^2_2, \qquad \low{\rho}=\min \rho(\cdot,t)>0,
\]
which implies $u \in L^\infty_t H^1_x \cap L^2_{t} H^2_x$.  In particular, this implies $|u_x(\cdot,t)|_\infty \in L^1$, and hence the estimates on $H^1$-norm of $e$ from Lemma~\ref{l:e}  closes with an integrable multiplier on the right hand side of \eqref{eq:leb}. It remains to establish a further similar bound on $|u''|_2$ to close the estimate on the grand quantity $|u''|_2^2 + |e'|_2^2 \sim |u''|_2^2 + |\rho|_{H^3}^2$. So, we differentiate \eqref{e:uagain} twice and test with $u''$:
\[
\begin{split}
\p_t |u''|_2^2 + \frac52 u'u'' u'' & = - (\rho u'')' u''' + 2 \rho''' u' u'' + 4 \rho'' u'' u'' + 2 \rho' u''' u'' \\
& = -\rho |u'''|^2 +  \rho' u''' u'' + 2 \rho''' u' u'' + 4 \rho'' u'' u''.
\end{split}
\]
So,
\[
\begin{split}
\p_t |u''|_2^2 & \lesssim |u_x|_\infty |u''|_2^2  - \frac12 \rho_- |u'''|_2^2 + |\rho'|_\infty |u''|_2^2 + |u_x|_\infty |\rho'''|_2 |u''|_2 + 4 | e - u_x|_\infty |u''|_2^2 \\
& \lesssim - \frac12 \rho_- |u'''|^2_2 +|u_x|_\infty (|e'|^2_2 + |u''|^2_2).
\end{split}
\]
Given the established integrability of $|u_x|_\infty$ and Lemma~\ref{l:e} we have proved boundedness in $H^2$ for $u$, and $H^3$ for $\rho$.

\section{Global existence: singular mollifiers $\aL=\aL_\a$, $1\leq \a < 2$}\label{sec:singularaction}
In this section we prove global regularity result for the equation with fractional $\aL = - \L_\a$ in space of data $H^3 \times H^{2+\a}$. The case $\a=1$ is critical similar to the classical fractional Burgers equation, \cite{KNS2008, CV2010, CV2012} but with additional non-linearity in the dissipation term.  We will leave the subcritical case $\a>1$ as an easy  consequence of the proof presented here for the case $\a=1$.  Note that with the initial datum $(u_0,\rho_0)$ in $H^3$ we can avoid making assumptions on $e_0$ as $e_0 \in H^2 \ss C^{1+\g}$ for any $\g<1/2$ by the Sobolev embedding. With this we recall a priori uniform bounds from  the previous section,
\begin{equation}\label{e:eru}
\sup_{0<t<T} | e|_\infty  < \infty, \quad 0< \low{\r}  \leq \rho(x,t) \leq \upp{\r}, \quad 
\low{u}  \leq u(x,t) \leq \upp{u}.
\end{equation}
on any finite time interval of existence. The lower bound on the density is  the main reason why we resort to the periodic domain. In the open space such bound is only known to hold on any finite interval, lacking a uniform parabolicity to the system.

Moreover, for a solution in $H^3$ the transport equation \eqref{e:tranQ} for $Q = e/\rho$ can be solved classically along characteristics of $u$ which results in the bound \eqref{e:control1} which we quote for convenience
\begin{equation}\label{e:high1}
|e_x(x,t)| \leq C |\rho_x(x,t)|, \text{ for all } (x,t) \in [0,T) \times \T.
\end{equation}
The proof will consist of four steps. First, we establish local existence in $H^3$ by obtaining rough a priori bounds without exploiting dissipation term. This allows to perform classical desingularization of the kernel as an approximate scheme to obtain local solutions. Second, we establish  uniform control over first order quantities $|\rho_x|_\infty$, $|u_x|_\infty$ over the interval of regularity. The strategy here  resembles the treatment of the critical SQG by Constantin and Vicol \cite{CV2012}, but with additional technicalities related to the non-linear nature of the dissipation term.  We then invoke the results of Schwab and Silverstre \cite{SS2012} to obtain instantaneous $C^\g$-regularization and use it to have an easier control on the oscillations in the midrange of scales of the non-linearity. Third, we establish uniform control over $H^2$ norm of solutions by proving an analogue of the Beale-Kato-Majda estimates. With the $H^2$ and $W^{1,\infty}$ bounds we finally conclude by a proving a uniform control of the penultimate $H^3$-norm of the solution on the entire interval of existence.

It will be useful to introduce the following notation. For three functions $f,g,h$ of $x,z$ we denote
\[
\F(f,g,h) := \frac12 \iint \frac{f(x,z)g(x,z)h(x,z)}{|z|^2} dz dx.
\]
Moreover, for a cutoff function $\f$ and parameter $r>0$ we denote
\[
\begin{split}
\F_{<r}(f,g,h) &= \frac12 \iint \frac{f(x,z)g(x,z)h(x,z)}{|z|^2} \f(z/r) dz dx \\
\F_{>r}(f,g,h) &= \frac12 \iint \frac{f(x,z)g(x,z)h(x,z)}{|z|^2} (1-\f(z/r)) dz dx.
\end{split}
\]
In the sequel we will also use the following notation $\d_z f(x) = f(x+z) - f(x)$, and the expansion
\begin{equation}\label{e:taylor}
\d_z f(x) = f'(x) z + z^2 \int_0^1 (1-\th) f''(x+ \th z) d\th.
\end{equation}

\subsection{Local well-posedness in $H^3$: a priori estimates without the use of dissipation} The purpose of this section is to obtain a priori estimates in $H^3$ which do not rely on the dissipation term. Namely, we will obtain the classical Riccati equation for the quantity $Y = |u|_{H^3} + |\rho|_{H^3} \sim |u|_{H^3} + |e|_{H^2} + |\r|_2$:
\[
Y_t \leq C Y^2,
\]
which is independent of desingularization of the kernel $K_\d = \frac{1}{(|z|^2 + \d^2)^{\frac{n+1}{2}}}$. This allows to conclude local existence via the classical approximation methods.

Let us write the equation for $u'''$:
\begin{equation}\label{e:u'''}
u'''_t + u u'''_x + 4u'u'''+ 3u'' u'' = \aT(\r''',u)+ 3 \aT(\r'',u') + 3\aT(\r',u'') + \aT(\r,u''').
\end{equation}
Testing with $u'''$ we obtain (we suppress integral signs and note that $\int u''u''u''' = 0$)
\begin{equation}\label{e:eb}
\p_t |u'''|_2^2 = - 7 u'(u''')^2 + 2 \aT(\r''',u)u''' +6 \aT(\r'',u')u''' + 6 \aT(\r',u''))u''' + 2\aT(\r, u''')u'''.
\end{equation}
We will now perform several estimates with the purpose of extracting term $|u'''|_2^2$ on the right hand side, times a lower order term in $u$ and possibly a top order term in $\r$ which we will address subsequently. First, we have trivially
\begin{equation}\label{e:b1}
|u'(u''')^2| \leq |u'|_\infty |u'''|_2^2.
\end{equation}
Let us estimate the dissipative term first:
\[
\int \aT(\r,u''')u''' dx = \iint \r(y) u'''(x) (u'''(y) - u'''(x)) \frac{dy\ dx}{|x-y|^2}.
\]
Switching $x$ and $y$ and adding cross-terms $\r(x)u'''(x)$ we obtain
\[
\begin{split}
\int \aT(\r,u''')u''' dx & = - \frac12 \iint \r(x) (u'''(y) - u'''(x))^2 \frac{dy\ dx}{|x-y|^2} \\
& + \frac12 \iint u'''(x)(\r(y) - \r(x)) (u'''(y) - u'''(x)) \frac{dy\ dx}{|x-y|^2}.
\end{split}
\]
The first term is clearly negative. We note in passing that it is bounded below by
\[
\iint \r(x) (u'''(y) - u'''(x))^2 \frac{dy\ dx}{|x-y|^2} \geq \urho |u'''|_{H^{1/2}}^2.
\]
While it is undoubtedly a crucial piece of information, it does depend on the fact that the kernel is singular. As we indicated earlier, however, we seek estimates that are independent of singularity.  So, at this point we will simply dismiss the dissipation term.  As to the second term, we rewrite it as
\[
\iint u'''(x)(\r(y) - \r(x)) (u'''(y) - u'''(x)) \frac{dy\ dx}{|x-y|^2} = \F( u''',\d_z\r, \d_z u''').
\]
We estimate the large-scale part of the integral using integrability of $|z|^{-2}$ at infinity as follows
\begin{equation}\label{}
\F_{>1}( u''',\d_z\r, \d_z u''') \leq |\r|_\infty |u'''|_2^2.
\end{equation}
As to the small scale, we use the expansion \eqref{e:taylor} on $\rho$. We have
\[
\begin{split}
\F_{<1}( u''',\d_z\r, \d_z u''') & =  
\iint \f(z) u'''(x) \r'(x) \d_z u'''(x) \frac{dz\ dx}{z} \\
&+
\int_0^1 (1-\th)  \iint \f(z) u'''(x) \r''(x+ \th z) \d_z u'''(x) dz\ dx \ d\th.
\end{split}
\]
Writing the first integral in the principal value sense results in the cancellation
\[
\iint \f(z) u'''(x) \r'(x) u'''(x) \frac{dz\ dx}{z} = 0,
\]
while
\[
\iint \f(z) u'''(x) \r'(x) u'''(x+z) \frac{dz\ dx}{z} = \int u'''(x) \r'(x) H_\f(u''')(x) dx,
\]
where $H_\phi$ is the truncated Hilbert transform given by convolution with the kernel $\frac{\f(z)}{z}$. It is a bounded operator on $L^2$. We thus have the estimate
\[
\left| \int u'''(x) \r'(x) H_\f(u''')(x) dx \right| \leq |\r'|_\infty |u'''|_2^2.
\]
Putting the estimates together, we arrive at the bound
\begin{equation}\label{e:b2}
\int \aT(\r,u''')u''' dx \leq  C |u'''|_2^2 ( |\r|_\infty + |\r'|_\infty).
\end{equation}
We now proceed with the remaining three terms in \eqref{e:eb} in a similar fashion. We have
\begin{equation*}\label{}
\begin{split}
\aT(\r''',u)u''' & = \F(\r'''(\cdot+z), u''', \d_z u)= \F_{>1}(\r'''(\cdot+z), u''', \d_z u)
+ \F_{<1}(\r'''(\cdot+z), u''', \d_z u)\\
& \leq | u'''|_2 |\r'''|_2 |u|_\infty + \int H_\f(\r''')(x) u'''(x) u'(x) dx \\
&+ \int_0^1 (1-\th) \iint \r'''(x+z) u'''(x) u''(x+\th z) \f(z) dz \ dx \ d\th\\
&\leq | u'''|_2 |\r'''|_2 (|u|_\infty + |u'|_\infty + |u''|_\infty).
\end{split}
\end{equation*}

\begin{equation*}\label{}
\begin{split}
\aT(\r'',u')u''' & = \F( \r''(\cdot +z), u''', \d_z u')= \F_{>1}( \r''(\cdot +z), u''', \d_z u')+ \int H_\f(\r'')(x) u'''(x) u''(x) dx \\
& + \int_0^1 (1-\th) \iint \r''(x+z) u'''(x) u'''(x+\th z) \f(z)\ dz\ dx \ d\th\\
& \leq | u'''|_2 |\r''|_2 |u'|_\infty + |u'''|_2|\r''|_2|u''|_\infty + | u'''|_2^2 |\r''|_\infty \\
& = | u'''|_2^2 |\r''|_\infty + | u'''|_2 |\r''|_2 (|u'|_\infty + |u''|_\infty).
\end{split}
\end{equation*}
And the last term requires more preparation,
\begin{equation*}\label{}
\begin{split}
\aT(\r',u'')u''' & = \iint \r'(y) u'''(x) (u''(y) - u''(x)) \frac{dy\ dx}{|x-y|^2} \\
&= \frac12 \iint (\r'(y) u'''(x) -\r'(x) u'''(y)) (u''(y) - u''(x)) \frac{dy\ dx}{|x-y|^2} \\
& = \frac12 \iint (\r'(y)-\r'(x)) u'''(x) (u''(y) - u''(x)) \frac{dy\ dx}{|x-y|^2} \\
&+ \frac12 \iint \r'(x) (u'''(x)-u'''(y))(u''(y) - u''(x)) \frac{dy\ dx}{|x-y|^2}\\
& = \frac12 \F( \d_z\r', u''', \d_z u'') - \frac12 \F(\r', \d_z u''',\d_z u'')\\
&= \frac12 \F_{>1}( \d_z\r', u''', \d_z u'') +
\frac12 \iint \f(z) \r''(x) u'''(x) \d_z u''(x) \frac{dz}{z}\ dx \\
&+ \frac12 \int_0^1 (1-\th) \iint \f(z) \r'''(x+\th z) u'''(x) \d_z u''(x) dz\ dx \ d\th \\
& - \frac14 \iint \r'(x)((\d_z u''(x))^2)' \frac{dz}{|z|^2}\ dx\\
& \leq |\r'|_\infty |u'''|_2 |u''|_2 + |\r''|_\infty |u'''|_2 |u''|_2 + |\r'''|_2 |u'''|_2 |u''|_\infty + \frac14 \iint \r''(x)(\d_z u''(x))^2 \frac{dz}{|z|^2}\ dx. 
\end{split}
\end{equation*}
The latter integral is bounded by $|\r''|_\infty |u''|_{\dot{H}^{1/2}}^2 \leq |\r''|_\infty |u'''|_2 |u''|_\infty$.
Putting the obtained estimates together we obtain
\begin{equation}\label{}
\begin{split}
\p_t |u'''|_2^2 &\leq C |u'''|_2^2(|u'|_\infty+|\r|_\infty + |\r'|_\infty+ |\r''|_\infty) \\
&+
| u'''|_2 (|\r''|_2+|\r'''|_2+|\r''|_\infty) (|u|_\infty + |u'|_\infty + |u''|_\infty) \\
&+ | u'''|_2|u''|_2(|\r'|_\infty+|\r''|_\infty).
\end{split}
\end{equation}
Finally, by Sobolev embedding, $|u''|_2 + |u|_\infty + |u'|_\infty + |u''|_\infty \leq C( |u'''|_2 + |u|_2)$, and $|\r|_\infty + |\r'|_\infty+ |\r''|_\infty + |\r''|_2 \leq C(|\r'''|_2 + |\r|_2)$ which results in the bound 
\begin{equation}\label{e:locu'''}
\p_t |u'''|_2^2 \lesssim (|u'''|_2 + |u|_2)^2 (|\r'''|_2 + |\r|_2 )  +(|u'''|_2 + |u|_2)^3.
\end{equation}
To control the energy $|u|_2$ we avoid using the natural balance relation \eqref{e:en}. Instead we test \eqref{e:main} directly with $u$. Performing much the same estimates as above we obtain, for example, 
\[
\p_t |u|^2_2 \leq |u|_\infty|u|_2 |\r|_2 + |\r'|_2|u'|_2 |u|_\infty.
\]
Putting this together with \eqref{e:locu'''} we obtain the Riccati equation for the $H^3$-norm:
\begin{equation}\label{}
\p_t |u|_{H^{3}} \leq |u|_{H^{3}}|\r|_{H^3} + |u|_{H^{3}}^2.
\end{equation}
In order to close the estimates we now have to find a similar bound on the $H^3$-norm of $\r$. This cannot be done directly by manipulating with the density transport equation. Instead we will make use of the transport of the first order quantity $e$, in terms of which we will provide the final estimates. Let us note the inequality
\[
|\r|_{H^3} \leq |u|_{H^3} + |e|_{H^2} + |\r|_2.
\]
Thus,
\begin{equation}\label{}
\p_t |u|_{H^{3}} \leq |u|_{H^{3}}(|e|_{H^2} +|\r|_2) + |u|_{H^{3}}^2.
\end{equation}
From Lemma~\ref{l:e} we have the bound on $|e|_{H^2}$:
\begin{equation}\label{e:eH3}
\p_t |e|_{H^2} \leq C (|e|_{H^2} + |u|_{H^3})^2.
\end{equation}
And the similar bound holds for $|\r|_2$. We have obtained the classical Riccati equation for the quantity $Y = |u|_{H^3} + |e|_{H^2} + |\r|_2$:
\[
Y_t \leq C Y^2.
\]
Note that $Y \sim |u|_{H^3} + |\r|_{H^3}$, hence we have proved necessary a priori bound for the local well-posedness in $H^3$.

\subsection{Control over $|u_x|_\infty$ and $|\rho_x|_\infty$ on intervals of regularity} Suppose that we have a classical solution $(u,\rho) \in C([0,T); H^{3})$ as proved to exist in the previous section. We now seek to establish a uniform bound on $|u_x|_\infty$ and $|\rho_x|_\infty$ on the entire interval $[0,T)$. First, let us recall that we have already established a priori uniform bounds of $e,\rho$ and $u$ in terms of the  finite initial quantities $e_\pm,\rho_\pm$ and $u_\pm$, consult \eqref{e:eru}.
Next, as we noted the density $\rho$ satisfies a parabolic form  of the density equation:
\begin{equation}
\rho_t+u\rho_x+e\rho= \rho \aL(\rho)
\end{equation}
Similarly, one can write the equation for the momentum $m = \rho u$:
\begin{equation}
m_t+u m_x+e m  = \rho \aL(m)
\end{equation}
Note that in both cases the drift $u$ and the forcing $e \rho$ or $e m$ are bounded a priori. Moreover, the diffusion operator has kernel
\[
K(x,h,t) = \rho(x) \frac{1}{|h|^2}
\]
which satisfies all the assumptions of Schwab and Silverstre \cite{SS2012}. A direct application of \cite{SS2012} tells us that there exists an $\g>0$ such that
\begin{equation}\label{e:gamma}
\begin{split}
| \rho|_{C^\g(\T \times [T/2,T))} &\leq C( |\r|_{L^\infty(0,T)} + |\rho e |_{L^\infty(0,T)})\\
| m |_{C^\g(\T \times [T/2,T))} &\leq C( |m|_{L^\infty(0,T)} + |me|_{L^\infty(0,T)}) \\
| u |_{C^\g(\T \times [T/2,T))} &\leq C( |u|_{L^\infty(0,T)}, |\rho|_{L^\infty(0,T)}) ,
\end{split}
\end{equation}
where the latter follows from the first two since $\rho$ is bounded below.
Of course, since $u,\rho$ are in $H^3$ on $[0,T)$ this implies $C^\g$-bound on the entire interval of regularity, however we need the bound to be independent of $H^3$, which may blow up, in the second half of it. It is also interesting to note that the original equation for $u$ has a kernel $K(x,h,t) = \rho(x+h)\frac{1}{|h|^2}$ not even with respect to $h$, so no known results on regularization are directly applicable to the $u$-equation.
\begin{remark} In regard to higher order regularization via Schauder, we make the following observation.  For $Q=e/\r$ we recall that  $Q_x$ was shown to be under control (note that this still doesn't imply that either $e_x$ or $\r_x$ are under control). Hence, trivially, $|Q|_{C^\g}$ remains bounded at all times.  Denote
\[
\d_h Q(x) = \frac{Q(x+h) - Q(x)}{|h|^\g},
\]
and note
\[
\d_h Q(x) = \frac{\d_h e(x)}{\r(x+h)} + \frac{e(x) \d_h \r(x)}{\r(x+h) \r(x)}.
\]
Since $\rho$ is $C^\g$ and bounded away from zero  this implies that $e\in C^\g$ with $|e(t)|_{C^\g} \leq C/t^\g$. With this in mind, we now have the momentum equation in the form
\[
m_t + b(x) m_x + a(x) \L m = F,
\]
where the drift $b$, the coefficient function $a$ and the source $F$ are all in $C^\g$. As of this writing there has been no known Schauder-type bounds proved for an equation in such generality despite many recent developments in that cover partial cases, see \cite{CK2015,DZ2015,SS2012,TJ2015}. The question presents an independent interest and we will address it in subsequent work.
\end{remark}

Let us now establish control over $\rho'$.  We write 
\[
\p_t \rho' + u \rho'' + u' \rho' + e'\rho + e \rho' = - \rho' \L \rho - \rho \L \rho'.
\]
Using again $u' = e + \L \rho$ we rewrite
\[
\p_t \rho' + u \rho'' + e'\rho + 2e \rho' = - 2\rho' \L \rho - \rho \L \rho'.
\]
Let us evaluate it at the maximum of $\rho'$ and multiply by $\rho'$ again (we use the classical Rademacher theorem here to justify the time derivative):
\begin{equation}\label{e:rho-dera}
\p_t |\rho'|^2 + e' \rho \rho' + 2 e |\rho'|^2 = -2 |\rho'|^2 \L \rho - \rho \rho' \L \rho'.
\end{equation}
In view of \eqref{e:eru} and \eqref{e:high1} we can bound
\[
|e' \rho \rho' + 2 e |\rho'|^2| \leq C |\rho'|^2.
\]
Next, using the nonlinear bounds from \cite{CV2012} we have 
\begin{equation}\label{e:disslower}
\rho \rho' \L \rho' \geq \frac14 \urho D\rho'(x) + c \frac{\urho}{\orho} |\rho'|_\infty^3 \geq c_1 D\rho'(x) + c_2 |\rho'|_\infty^3.
\end{equation}
where 
\[
D\rho'(x) = \int_\R  \frac{|\rho'(x) - \rho'(x+z)|^2}{|z|^2} dz.
\]
Using smooth decompositions of the underlying $\R$ in all of the below we have 
\[
\begin{split}
\L \rho(x)  = H \rho' &= \int_{|z| <r} \frac{\rho'(x+z) - \rho'(x)}{z} dz - \int_{r<|z| <2\pi} \frac{\rho(x+z) - \rho(x)}{|z|^2} dz \\
&-\int_{2\pi <|z|} \frac{\rho(x+z) - \rho(x)}{|z|^2} dz .
\end{split}
\]
The latter is clearly bounded by a constant $c_3$ depending only on $\urho$, which in \eqref{e:rho-dera} results simply in the bound $c_3 |\rho'|_\infty^2$.
The first is bounded, via H\"older, by 
\[
 |\rho'|_\infty^2 \sqrt{r} D^{1/2}\rho' (x) \leq \frac12 c_1D\rho'(x) +  c_4 r  |\rho'|_\infty^4.
\]
Note that this  term gets absorbed by the dissipation  \eqref{e:disslower} entirely if
\[
r = \frac{c_2}{4 c_4 |\rho'|_\infty}.
\]
The integral in the middle is bounded by, using $C^\g$-regularity,
\[
 |\rho'|_\infty^2 |\r|_{C^\g} / r^{1-\g} = c_5  |\rho'|_\infty^{3-\g} \leq c_6 + \frac{c_2}{4}  |\rho'|_\infty^3,
\]
where the cubic term again is absorbed by the dissipation. Putting the estimates together we obtain
\begin{equation}\label{e:rho-derb}
\p_t |\rho'|^2 \leq c_6 + c_3 |\rho'|^2 - c_7 D \rho'(x),
\end{equation}
which establishes the claimed control of $\rho'$. We intentionally keep the dissipation  term as it still will be used on the next step to absorb other terms. 

Now we can do the same for the momentum derivative $m_x$. Clearly it is sufficient to finish the proof for $u_x$ as well.  Note that the equation for momentum is similar, so we will skip details that are similar.  We have
\[
\p_t m' + u m'' + u' m' + e'm + e m' = - \rho' \L m - \rho \L m'.
\]
Evaluating at maximum, multiplying by $m'$, and using bounds on $e,e'$ we have
\begin{equation}\label{e:mprime}
\p_t |m'|^2 \leq c_8 ( |m'|^2_\infty + |\rho'|_\infty) + |m'|^2 |\L \rho|   + |\rho'| |m'| |\L m| -  c_9 Dm'(x) - c_{10} |m'|_\infty^3.
\end{equation}
As to $ |\rho'| |m'| |\L m|$ we proceed as before, loosing $\rho'$ in view of already established control over it. We obtain the bound simply by taking $r = 1$:
\[
c_{11}  |m'| +  |m'| D^{1/2}m' (x) \leq c_{11} |m'| + c_{12} |m'|^2 + \frac{c_9}{4}Dm'(x),
\]
with the latter being absorbed again in the dissipation. As to the term $ |m'|^2 |\L \rho| $ we still proceed as before, however in the mid-range integral $r<|z|<2\pi$ we use the full force of the obtained bound on $\rho'$. This results in logarithmic optimization bound
\[
|m'|^2 |\L \rho| \leq  c_{13} |m'|^2(1 + \ln r + \sqrt{r} D^{1/2} \rho'(x) ).
\]
Ignoring the trivial quadratic term $|m'|^2$ we have
\[
c_{13} |m'|^2 \ln r + c_{13} |m'|^2  \sqrt{r} D^{1/2} \rho'(x) \leq c_{13} |m'|^2 \ln r + c_{14}|m'|^4 r +  \frac{c_7}{2} D \rho'(x).
\]
Notice that the latter will be absorbed by the dissipation term in \eqref{e:rho-derb} when we add the two equations together. Choosing 
\[
r = \frac{c_{10}}{2 c_{14} |m'|},
\]
we obtain for the $\ln r$ and $r$-terms above the bound
\[
c_{15} |m'|^2 \ln |m'| + \frac{c_{10}}{2} |m'|^3,
\]
with the latter being absorbed into the cubic term in \eqref{e:mprime}. Altogether we have
\begin{equation}\label{e:mprime2}
\p_t |m'|^2 \leq c_{16} |m'|^2 (1+ \ln_+ |m'|) +\frac{c_7}{2} D \rho'(x).
\end{equation}
We now have to add the two equations \eqref{e:mprime2} and \eqref{e:rho-derb} together to absorb the residual $D\rho'$-term and obtain the final bound
\begin{equation}\label{eq:double}
\p_t (|m'|^2 + |\rho'|^2) \leq c_{17}  (|m'|^2 + |\rho'|^2) (1+ \ln_+ (|m'|^2 + |\rho'|^2)).
\end{equation}
This implies double-exponential, but finite,  bound on the given interval. This also finishes the proof.

\subsection{Control over $H^2$ via $|u_x|_\infty$}\label{s:H2} 
In this section we will establish an estimate on the $H^2$-norm of the solution 
\[
X = |u''|^2_2 +|\r''|_2^2 \sim   |u''|^2_2 + |e'|_2^2
\]
in terms of $|u_x|_\infty$ is a manner similar to the Beale-Kato-Majda criterion. 
Namely, we will prove
\begin{equation}\label{e:log}
X' \leq C(1 + |u'|_\infty) X (1+ \log_+ X).
\end{equation}
Given the result of the previous section this establishes uniform bound in $H^2$ on the interval of existence $[0,T)$ of an $H^3$-solution.
The equation for $u''$ reads
\[
u''_t + u u''_x + 3u' u''  = \aT(\r'',u)+2 \aT(\r',u') + \aT(\r,u'').
\]
Testing with $u''$ the local terms , after integration by parts , become bounded by $X |u'|_\infty$ trivially. We now look into key estimates for the right hand side. We will start with what proved to be the most involved term in the previous section. We skip the standard symmetrization and addition of cross-product terms in the calculations below and typically display the final representations. We have
\begin{equation}\label{e:Tuu}
\iint \aT(\r', u') u'' dx dy = \F(\d_z \r', \d_z u' , u'') + \F(\r', \d_z u', \d_z u'').
\end{equation}
For the second term we have $\d_z u' \d_z u'' = \frac12 ( (\d_z u')^2)_x$. So, switching the derivative onto $\r'$ we obtain
\[
\F (\r', \d_z u', \d_z u'') = - \frac12 \F(\r'', \d_z u',\d_z u').
\]
Now, we bound small and large scale parts as follows
\[
 | \F_{>r} (\r'', \d_z u',\d_z u') |  \leq \frac{1}{r} | \r''|_2 |u'|_4^2,
 \]
and
\[
  | \F_{< r} (\r'', \d_z u',\d_z u') |  \leq  \sqrt{r} | \r''|_2 |u' |_{W^{3/4,4}}^2,
\]
where in the latter we used the H\"older and Gagliardo-Sobolevskii definition of $W^{3/4,4}$ space.  Optimizing over $r$ we obtain
\[
| \F (\r', \d_z u', \d_z u'')  | \leq | \r'' |_2 |u'|_4^{2/3} |u'|_{W^{3/4,4}}^{4/3},
\]
and by Gagliardo-Nirenberg,
\[
 |u'|_{W^{3/4,4}} \leq |u ''|_{H^{1/2}}^{1/2} |u'|_\infty^{1/2},
\]
and interpolation we obtain
\[
| \F (\r', \d_z u', \d_z u'')  | \leq  | \r'' |_2   |u'|_4^{2/3} |u' |_\infty^{2/3} |u'' |_{H^{1/2}}^{2/3} \leq  \frac{1}{\e}  | \r'' |_2^{3/2}   |u'|_4 |u' |_\infty + \e |u'' |_{H^{1/2}}^{2} .
\]
With $\e <\urho/2$ the last term is absorbed by the dissipation. Finally, by Gagliardo-Nirenberg we have
\begin{equation}\label{e:GN2}
|u' |_4 \leq |u''|_2^{1/2} |u|_\infty^{1/2}.
\end{equation}
Recalling that $|u|_\infty$ is under control by the maximum principle, we finally obtain
\[
| \F (\r', \d_z u', \d_z u'')  | \leq  C  | \r'' |_2^{3/2} |u''|_2^{1/2}|u' |_\infty+ \e |u'' |_{H^{1/2}}^{2} \leq  C |u' |_\infty X + \e |u'' |_{H^{1/2}}.
\]
For the other term $\F(\d_z \r', \d_z u' , u'') $ the splitting is necessary but optimization is not. We have, in view of \eqref{e:GN2},
\[
\F_{>1} (\d_z \r', \d_z u' , u'')  \leq |\r'|_4|u'|_4 |u''|_2 \lesssim |\r''|_2^{1/2} |u''|_2^{3/2} \leq X.
\]
As to $\F_{<1}$, we write $\d_z u'(x) = \d_z u'(x) - z u'(x) + zu'(x)$, and note that $|\d_z u'(x) - z u'(x)| \leq |z|^2 |u''(x+\th z)|$, for some $\th$, which is unimportant. So, we have
\[
|\F_{<1} (\d_z \r', \d_z u' , u'')| \leq \left| \int u'(x) H\r'(x) u''(x) dx \right| + |\r'|_\infty |u''|_2^2 \leq ( |\r'|_\infty + |u'|_\infty) X.
\]
Note that 
\[
 |H \r'|_\infty \leq |e|_\infty + |u'|_\infty,
\]
and by the  log-Sobolev inequality,
\[
 |\r'|_\infty \leq  |H \r'|_\infty(1+ \log_+ |\r''|_2)\leq (C + |u'|_\infty) (1+ \log_+ X).
 \]
So,  
\[
|\F_{<1} (\d_z \r', \d_z u' , u'')| \leq  (C + |u'|_\infty) X (1+ \log_+ X).
\]
We have proved the bound
\[
| \aT(\r', u') u'' | \leq  (C + |u'|_\infty) X (1+ \log_+ X) +  \e |u'' |_{H^{1/2}}.
\]

Next, let us bound the dissipation term
\[
\iint \aT(\r,u'') u'' \, dy dx =  - \F(\r , \d_z u'', \d_z u'') + \F(\d_z \r, \d_z u'', u'').
\]
Obviously,
\[
 \F(\r , \d_z u'', \d_z u'') \geq \urho |u''|_{H^{1/2}}^2.
\]
As to  $\F(\d_z \r, \d_z u'', u'')$ we have
\[
\F_{>r}(\d_z \r, \d_z u'', u'') \leq \frac{1}{r} |u''|_2^2,
\]
and 
\[
\F_{<r}(\d_z \r, \d_z u'', u'') \leq  |\r'|_\infty \int_\R |u''(x) | \int_{|z|<2r}\left|  \frac{\d_z u''}{z} \right| dz dx \leq  \sqrt{r} |\r'|_\infty |u''|_2 |u''|_{H^{1/2}}.
\]
Optimizing we obtain
\[
\begin{split}
 \F(\d_z \r, \d_z u'', u'') & \leq |\r'|_\infty^{2/3} |u''|_2^{4/3} |u''|_{H^{1/2}}^{2/3} \leq \e |u''|_{H^{1/2}}^{2}  + \frac{1}{\e} |\r'|_\infty |u''|_2^{2} \\
 &\leq  \e |u''|_{H^{1/2}}^{2}  + (C + |u'|_\infty) X (1+ \log_+ X),
 \end{split}
 \]
 which closes the estimates with the help of dissipation.

It remains to estimate the last term. By switching $x$ and $y$ we obtain
\begin{equation}\label{e:star}
\begin{split}
\iint \aT(\r'',u) u'' \, dy dx &= \iint \r''(x)(u(x)-u(y)) u''(y) \frac{dy dx}{|x-y|^2} \\
&=  \iint \r''(x)(u(x)-u(y))( u''(y) -u''(x)) \frac{dy dx}{|x-y|^2}  \\
&+  \iint \r''(x)u''(x) (u(x)-u(y)) \frac{dy dx}{|x-y|^2} \\
&= \F(\r'',\d_z u, \d_z u'') + \int \r'' u'' \L(u) dx.
\end{split}
\end{equation}
Clearly, by the log-Sobolev inequality,
\[
| \int \r'' u'' \L(u) dx | \leq |\r''|_2 |u''|_2 |\L u|_\infty \lesssim |u'|_\infty X (1+ \log_+ X).
\]
For the F-term we have
\[
| \F_{>r} (\r'',\d_z u, \d_z u'') | \leq \frac{1}{r} |\r''|_2 |u''|_2,
\]
while 
\[
| \F_{<r} (\r'',\d_z u, \d_z u'') | \leq |u'|_\infty \int |\r''(x)| \int_{|z|<2r} \frac{|\d_z u''(x)|}{|z|} dz dx \leq |u'|_\infty \sqrt{r} | \r''|_2 | u'' |_{H^{1/2}}.
\]
Optimizing, we get
\[
|\F (\r'',\d_z u, \d_z u'')| \leq  | \r''|_2 |u''|_2^{1/3} |u'|_\infty^{2/3} |u''|_{H^{1/2}}^{2/3} \leq \e |u''|_{H^{1/2}}^2 + \frac{1}{\e}  | \r''|_2^{3/2} |u''|_2^{1/2} |u'|_\infty \leq  |u'|_\infty X.
\]
We have proved that 
\[
\p_t |u ''|_2^2 \leq  - \e |u''|_{H^{1/2}}^{2}  + C(1 + |u'|_\infty) X (1+ \log_+ X).
\]
As to quantity $e$, we apply Lemma~\ref{l:e} to obtain
\[
\p_t |e'|_2^2 \leq C(1 + |u'|_\infty) X.
\]
Putting the estimates together, \eqref{e:log} follows.

\subsection{Control over $H^3$ via $H^2$ and $|u_x|_\infty$} For a given classical solution $(u,\rho) \in C([0,T); H^{3})$ we have established uniform bounds on $|u_x,\rho_x|_\infty$ and $|u,\rho|_{H^2}$ on the entire interval $[0,T)$. We now seek to establish final control over the $H^3$-norms. Note that we already have estimate \eqref{e:eH3} which with the new information readily implies
\[
\p_t |e''|_2^2 \lesssim |e''|_2^2 + |u'''|_2^2.
\]
Now we get to bounds on $|u'''|_2^2$. Not surprisingly all of the estimates mimic the already obtained sharper estimates for $H^2$ with the use of dissipation. In what follows we will indicate necessary changes and refer to appropriate places in Section~\ref{s:H2} for details. Also, we will drop from the estimates all quantities that are already known to be bounded, such as $|u,\rho|_{H^2}$, etc.  Thus, following \eqref{e:u'''} we can see that all the terms on the left hand side obey the bound by $|u'|_\infty |u'''|_2^2 \lesssim |u'''|_2^2$. We are left with the four terms on the right hand side:
\[
\aT(\r''',u)u''', \quad \aT(\r'',u')u''', \quad  \aT(\r',u'')u''', \quad \aT(\r, u''')u'''.
\]
First, the dissipation term obeys the same bound  \eqref{e:b2} where we now keep the dissipation :
\begin{equation}\label{e:b3}
\int \aT(\r,u''')u''' dx \leq  - \urho |u'''|_{H^{1/2}}^2 + C |u'''|_2^2 ( |\r|_\infty + |\r'|_\infty) \lesssim  - \urho |u'''|_{H^{1/2}}^2 +  C |u'''|_2^2.
\end{equation}
Next, the term $\aT(\r''',u)u'''$ will be estimated in the same way as \eqref{e:star} with replacements $\rho'' \to \rho'''$, $u'' \to u'''$.  We have the bound
\[
|\aT(\r''',u)u'''| \leq |\r'''|_2 |u'''|_2 |\L u|_\infty + \frac{\urho}{10} |u'''|_{H^{1/2}}^2 + \frac{10}{\urho}  | \r'''|_2^{3/2} |u'''|_2^{1/2} .
\]
Since $ |\L u|_\infty \leq |u|_{H^2} < C$ and $|\rho'''|_2 \leq |e''|_2 + |u'''|_2$ we have
\[
|\aT(\r''',u)u'''| \leq  |e''|^2_2 + |u'''|_2^2 +
\frac{\urho}{10} |u'''|_{H^{1/2}}^2.
\]
Next, the term $\aT(\r'',u')u'''$ will also be estimates as in \eqref{e:star} with a simple replacement $u \to u'$, i.e. raising the derivative of $u$ by one on every step. We obtain directly, 
\[
|\aT(\r'',u')u''' |  \leq  |\r''|_2 |u'''|_2 |\L u'|_\infty  +  \e |u'''|_{H^{1/2}}^2 + \frac{1}{\e}  | \r''|_2^{3/2} |u'''|_2^{1/2} |u''|_\infty.
\]
Dropping $|\rho''|_2$ and using that $ |\L u'|_\infty,  |u''|_\infty \leq |u|_{H^3}$, we obtain
\[
|\aT(\r'',u')u''' |  \leq \frac{\urho}{10} |u'''|_{H^{1/2}}^2 +  |u|_{H^3}^2 \leq \frac{\urho}{10} |u'''|_{H^{1/2}}^2 + C +  |u'''|_{2}^2.
\]
Finally, the term $ \aT(\r',u'')u'''$ can be estimates as term \eqref{e:Tuu} by raising the derivative of $u$ by one and with the use of boundedness of $|\rho''|_2$, $|\rho'|_\infty$. We obtain
\[
\begin{split}
| \aT(\r',u'')u'''| & \leq   \e |u''' |_{H^{1/2}}^{2} +  \frac{1}{\e}  | \r'' |_2^{3/2}   |u''|_4 |u'' |_\infty +  |\r''|_2^{1/2} |u'''|_2^{3/2}  \\
& +  \left| \int u''(x) H\r'(x) u'''(x) dx \right| + |\r'|_\infty |u'''|_2^2 .
\end{split}
\]
We have trivially, $  |u''|_4 |u'' |_\infty \leq |u|_{H^3}^2$, and
\[
 \left| \int u''(x) H\r'(x) u'''(x) dx \right|  \leq |u''|_2 |u'''|_2 |H\rho'|_\infty \lesssim  |u'''|_2 |H\rho''|_2 \lesssim |u'''|_2.
 \]
This completes the estimate for the $H^3$-norm $Y$:  $Y' \leq C Y$ on the time interval of existence. This completes the proof.


\end{document}